\documentclass[leqno,11pt]{paper}
\usepackage[utf8]{inputenc}

\usepackage{amsmath,amssymb,amsthm}
\usepackage{amscd}
\usepackage{setspace} 

\usepackage{enumerate}
\usepackage{array}
\usepackage{fancyhdr,ifthen}
\usepackage{tabularx}
\usepackage{multirow}
\usepackage{booktabs}

\theoremstyle{definition}
\newtheorem{definition}{Definition}

\newtheorem*{remarkson}{Remarks}

\theoremstyle{plain}

\newtheorem{theorem}{Theorem}

\newtheorem{lemma}[definition]{Lemma}

\newtheorem{corollary}{Corollary}

\theoremstyle{remark}

\newcommand{\C}{\mathbb{C}}

\newcommand{\G}{\mathbb{G}}

\newcommand{\N}{\mathbb{N}}
\newcommand{\Pb}{\mathbb{P}}
\newcommand{\Q}{\mathbb{Q}}
\newcommand{\R}{\mathbb{R}}
\newcommand{\Z}{\mathbb{Z}}

\newcommand{\Gcal}{\mathcal{G}}
\newcommand{\Hcal}{\mathcal{H}}

\newcommand{\Tcal}{\mathcal{T}}

\newcommand{\diag}{\operatorname{diag}\,}

\newcommand{\Sp}{\operatorname{Sp}}

\let\leq\leqslant
\let\geq\geqslant

\newenvironment{lcase}{\left\lbrace \begin{aligned}}{\end{aligned}\right.}

\begin{document}

\begin{center}
\begin{huge}
\begin{spacing}{1.0}
\textbf{The Hecke algebras for the orthogonal group $SO(2,3)$ and the paramodular group of degree $2$}  
\end{spacing}
\end{huge}

\bigskip
by
\bigskip

\begin{large}
\textbf{Jonas Gallenkämper\footnote{Jonas Gallenkämper, Lehrstuhl A für Mathematik, RWTH Aachen, D-52056 Aachen, jonas.gallenkaemper@mathA.rwth-aachen.de}},
\textbf{Aloys Krieg\footnote{Aloys Krieg, Lehrstuhl A für Mathematik, RWTH Aachen, D-52056 Aachen, krieg@rwth-aachen.de}}
\end{large}
\vspace{1cm}\\
12 March 2018
\vspace{1cm}
\end{center}
\begin{abstract}
In this paper we consider the integral orthogonal group with respect to the quadratic form of signature $(2,3)$ given by 
$\left(\begin{smallmatrix}
   0 & 1 \\ 1 & 0                                                                          
 \end{smallmatrix}\right) \perp 
\left(\begin{smallmatrix}
   0 & 1 \\ 1 & 0                                                                          
 \end{smallmatrix}\right) \perp (-2N)$ for squarefree $N\in \N$. The associated Hecke algebra is commutative and the tensor product of its primary components, which turn out to be polynomial rings over $\Z$ in $2$ algebraically independent elements.  \\
 The integral orthogonal group is isomorphic to the paramodular group of degree $2$ and level $N$, more precisely to its maximal discrete normal extension. The results can be reformulated in the paramodular setting by virtue of an explicit isomorphism. The Hecke algebra of the non-maximal paramodular group inside $\mathrm{Sp}(2;\Q)$ fails to be commutative if $N> 1$.
\end{abstract}

\newpage
\section{Introduction}

The Hecke theory plays an important role in the arithmetic theory of modular forms. Dealing with Siegel modular forms, a structure theorem was derived by Shimura \cite{Sh} (cf. \cite{An}, \cite{Fr}). Considering paramodular groups, such a result does not seem to be known. Dealing with orthogonal groups, authors (cf. \cite{HS}, \cite{Su}) usually define their Hecke algebra as the tensor product of local components. On the other hand, several authors have investigated the connection between the paramodular group and the orthogonal group $SO(2,3)$ (cf. \cite{G}, \cite {GH1}, \cite{GH2}, \cite{RS}).

In this paper we derive a structure result for the Hecke algebra associated with $SO(2,3)$. By means of an explicit isomorphism we obtain analogous results for the paramodular group of degree $2$ and squarefree level $N$. Moreover, we show that the Hecke algebra for the non-maximal paramodular subgroup inside $\mathrm{Sp}_2(\Q)$ also coincides with the tensor product of its primary components, but fails to be commutative similar to the case of the Fricke groups in \cite{K2}. Most of these results are contained in \cite{Ga}, where also results on the Hecke theory for $O(2,n+2)$ are given.

Let us fix some notation. Let $(U,G)$ be a Hecke pair, i.e. $U$ is a subgroup of $G$ and each double coset $UgU$, $g\in G$, decomposes into finitely many right cosets $Uh$, $h\in G$. Denote by $\Hcal(U,G)$ the Hecke algebra over $\Z$ of $(U,G)$ just as in \cite{An}, \cite{Fr}, \cite{K1}, \cite{Sh}.

\section{The orthogonal group $SO(2,3)$}

We fix further notation throughout the whole paper. Given $N\in\N$ let 
\begin{gather*}\tag{1}\label{gl_1}
 S_N = \begin{pmatrix}
        0 & 0 & 1  \\  0 & -2N & 0  \\  1 & 0 & 0
       \end{pmatrix}, \quad 
 \widehat{S}_N = \begin{pmatrix}
        0 & 0 & 1  \\  0 & S_N & 0  \\  1 & 0 & 0
       \end{pmatrix}    
\end{gather*}
be even symmetric matrices of signature $(1,2)$ resp. $(2,3)$. Consider the attached special orthogonal groups
\begin{align*}
 SO(S_N;\R):&=\{K\in SL_3(\R);\;S_N[K] = S_N\},  \\
 SO(\widehat{S}_N;\R):&=\{M\in SL_5(\R);\;\widehat{S}_N[M] = \widehat{S}_N\}, 
\end{align*}
where the prime stands for the transpose and $A[B]: = B'AB$ for matrices $A,B$ of suitable size. Let $SO_0(S_N;\R)$ resp. $SO_0(\widehat{S}_N;\R)$ stand for the connected component of the identity matrix $E$ (of suitable size), which was characterized in \cite{K3}, sect. 2. For the integers instead of the reals we use the analogous notations for the subgroups, in particular
\[
 \Gamma_N:=SO_0(S_N;\Z), \quad \widehat{\Gamma}_N:= SO_0(\widehat{S}_N;\Z).
\]
Just as in \cite{K3} we consider the matrices
\begin{gather*}\tag{2}\label{gl_2}
 M_{\lambda}:= \begin{pmatrix}
                1 & -\lambda'S_N & -\frac{1}{2}S_N[\lambda]  \\
                0 & E & \lambda  \\
                0 & 0 & 1
               \end{pmatrix}, 
\end{gather*}              
\begin{gather*}\tag{2'}\label{gl_2'}               
 \widetilde{M}_{\lambda}:= \begin{pmatrix}
                1 & 0 & 0  \\
                \lambda & E & 0  \\
                -\frac{1}{2}S_N[\lambda] & -\lambda'S_N & 1
               \end{pmatrix},\quad\lambda\in\R^3,
\end{gather*}
\begin{gather*}\tag{3}\label{gl_3}
 J^{\ast}:= \begin{pmatrix}
      0 & 0 & -1 \\ 0 & V & 0 \\ -1 & 0 & 0
     \end{pmatrix}, \quad 
 V = \begin{pmatrix}
      0 & 0 & -1 \\ 0 & 1 & 0 \\ -1 & 0 & 0
     \end{pmatrix},    
\end{gather*}
\begin{gather*}\tag{4}\label{gl_4}
 M_F:= \begin{pmatrix}
      F & 0 & 0 \\ 0 & 1 & 0 \\ 0 & 0 & F^{\ast}
     \end{pmatrix}, \;\; 
 F = \begin{pmatrix}
      \alpha & \beta \\ \gamma & \delta
     \end{pmatrix} \in SL_2(\R), \;\;
 F^{\ast} = \begin{pmatrix}
      \alpha & -\beta \\ -\gamma & \delta
     \end{pmatrix},
\end{gather*}
\begin{gather*}\tag{4'}\label{gl_4'}
 \widetilde{M}_F:= \begin{pmatrix}
      \alpha E & 0 & \beta I \\ 0 & 1 & 0 \\ \gamma I & 0 & \delta E
     \end{pmatrix}, \;\;
 F = \begin{pmatrix}
      \alpha & \beta \\ \gamma & \delta
     \end{pmatrix} \in SL_2(\R), \;\;
 I = \begin{pmatrix}
      -1 & 0 \\ 0 & 1
     \end{pmatrix},
\end{gather*}
\begin{gather*}\tag{5}\label{gl_5}
 \widehat{K}:= \begin{pmatrix}
      1 & 0 & 0 \\ 0 & K & 0 \\ 0 & 0 & 1
     \end{pmatrix}, \quad K\in SO_0(S_N;\R)
\end{gather*}
in $SO_0(\widehat{S}_N;\R)$ as well as 
\begin{gather*}\tag{6}\label{gl_6}
 K_{\mu}:= \begin{pmatrix}
      1 & 2N\mu & N\mu^2 \\ 0 & 1 & \mu \\ 0 & 0 & 1
     \end{pmatrix}, \quad 
 \widetilde{K}_{\mu}:= \begin{pmatrix}
      1 & 0 & 0 \\ \mu & 1 & 0\\ N\mu^2 & 2N\mu & 1
     \end{pmatrix}, \;\; \mu\in\R,
\end{gather*}
in $SO_0(S_N;\R)$.

\begin{lemma}\label{Lemma_1} 
 Let $N\in\N$ be squarefree.
\begin{enumerate}[a)]
\item Given $g\in\Z^3$ with $S_N[g] = 0$, there exists a matrix $K\in \Gamma_N$ such that 
\[
 Kg = (\gamma,0,0)', \;\; \gamma=\pm gcd(g).
\]
\item Given $g\in\Z^5$ with $\widehat{S}_N[g] = 0$, there exists a matrix $M\in \widehat{\Gamma}_N$ such that
\[
 Mg = (\gamma,0,0,0,0)', \;\; \gamma=gcd(g).
\]
\end{enumerate}
\end{lemma}

\begin{proof}
a) Apply Corollary 4 and (15) in \cite{K3}.  \\
b) Multiply by a suitable matrix $M_F$, $F\in SL_2(\Z)$, in \eqref{gl_4} in order to assume $g=(\ast,\ast,\ast,\ast,0)'$ without restriction. Then use a) and \eqref{gl_5} in order to obtain $g=(\ast,\ast,0,0,0)'$.  Finally, we get the result by a matrix $M_G$, $G\in SL_2(\Z)$, from \eqref{gl_4}.
\end{proof}

Lemma \ref{Lemma_1} does not hold for arbitrary $N\in \N$ as the example 
\[
 N = 4, \;\; g = (2,1,2)'
\]
in a) shows, because the first and last entry of $Kg$, $K\in\Gamma_N$, are always even.
\bigskip

We use the results in order to obtain suitable representatives for right cosets in 
\[
 \mathcal{G}_N:= SO_0(S_N;\Q), \quad \widehat{\mathcal{G}}_N:=SO_0(\widehat{S}_N;\Q).
\]

\begin{lemma}\label{lemma_2} 
Let $N\in\N$ be squarefree.
\begin{enumerate}[a)]
\item Let $\frac{1}{m}K\in\mathcal{G}_N$ with $m\in\N$ and integral $K$. Then the right coset $\Gamma_N K$ contains a unique representative of the form
\begin{gather*}\tag{7}\label{gl_7}
  L = \begin{pmatrix}
      \alpha^{\ast} & 2Nm\mu/\delta^{\ast} & N\mu^2/\delta^{\ast}  \\
      0 & m & \mu  \\
      0 & 0 & \delta^{\ast}
     \end{pmatrix}, \;\; 
     \begin{matrix}
     \alpha^{\ast},\delta^{\ast}\in\N,\;\;\alpha^{\ast}\delta^{\ast}=m^2, \\[1ex]
      \hspace*{-2ex}\mu\in\{0,1,\ldots,\delta^{\ast}-1\} . \qquad
      \end{matrix}
\end{gather*}
\item Let $\frac{1}{m} M\in \widehat{\mathcal{G}}_N$ with $m\in\N$ and integral $M$. Then the right coset $\widehat{\Gamma}_N M$ contains a unique representative of the form
\begin{gather*}\tag{8}\label{gl_8}
 \begin{pmatrix}
  \alpha & a' & \beta  \\ 0 & L & c  \\  0 & 0 & \delta
 \end{pmatrix},\;\; 
 \begin{matrix}
 \alpha,\delta\in \N,\;\alpha\delta = m^2,\;c\in\{0,1,\ldots,\delta-1\}^3,\\[1ex]
  \hspace*{-13ex}\beta=\frac{-1}{2\delta}S_N[c],\;a=-\frac{1}{\delta} L'S_N c,
  \end{matrix}
\end{gather*}
where $L$ has the form \eqref{gl_7} and $\alpha$ is the $gcd$ of the first column of $M$.
\end{enumerate}
\end{lemma}

\begin{proof}
a) For the existence apply Lemma 1 and \eqref{gl_6} with an appropriate $\mu\in\Z$. As block triangular matrices form a group, we obtain the last row in the form $(0,0,\ast)$ from the explicit form of the inverse in \cite{K3}. As this matrix belongs to the connected component of $E$, we get $\alpha^{\ast},\delta^{\ast}\in\N$. Given two matrices in the right coset of such a form multiply by the inverse of one from the right. Then one successively sees that they need to be equal. \\
b) This part follows along the same lines using Lemma 1, a), \eqref{gl_5} and \eqref{gl_2} with an appropriate $\lambda \in\Z^3$. The shape of the first row is a consequence of $\widehat{S}_N[M] = m^2\widehat{S}_N$ (cf. \cite{K3}).
\end{proof}

A necessary condition that the matrices in \eqref{gl_7} and \eqref{gl_8} occur is that they are integral, which is not the case for freely chosen integral $\alpha, \alpha^{\ast},\mu,c$. Lemma \ref{lemma_2} can be helpful in order to compute canonical representatives of right cosets in a given double coset or only for its calculation.\\

A direct consequence is

\begin{corollary} 
Let $N\in\N$ be squarefree. Then $(\Gamma_N,\mathcal{G}_N)$ and $(\widehat{\Gamma}_N,\widehat{\mathcal{G}}_N)$ are Hecke pairs.
\end{corollary}

Next we consider double cosets.

\begin{lemma} 
Let $N\in\N$ be squarefree and $\frac{1}{m}K\in\mathcal{G}_N$ with $m\in\N$ and integral $K$. Then the double coset $\Gamma_N K \Gamma_N$ contains a unique representative of the form
\[
 \diag(\alpha^{\ast},m,\delta^{\ast}), \;\alpha^{\ast},\delta^{\ast}\in\N,\;\alpha^{\ast}\delta^{\ast}=m^2,\;\alpha^{\ast}\mid m\mid \delta^{\ast}.
\]
The double coset is uniquely determined by the Smith invariants of $K$. 
\end{lemma}

\begin{proof}
The result follows from Theorem 4 in \cite{K2} in combination with Corollary 4 in \cite{K3}.
\end{proof}

Finally, we consider $(\widehat{\Gamma}_N, \widehat{\mathcal{G}}_N)$.

\begin{theorem}\label{Theorem_1} 
Let $N\in\N$ be squarefree. Each double coset $\widehat{\Gamma}_N(\frac{1}{m}M)\widehat{\Gamma}_N$, $\frac{1}{m}M \in\mathcal{G}_N$ with $m\in\N$ and integral $M$, contains a unique representative of the form
\[
 \frac{1}{m}\diag(\alpha,\alpha^{\ast},m,\delta^{\ast},\delta),
\]
where
\[
 \alpha,\alpha^{\ast},\delta,\delta^{\ast} \in \N,\;\; \alpha \delta = \alpha^{\ast}\delta^{\ast} = m^2,\;\; \alpha\mid\alpha^{\ast}\mid m\mid \delta^{\ast}\mid \delta.
\]
The double coset is uniquely determined by the Smith invariants of $M$.
\end{theorem}

\begin{proof}
Let $\alpha$ be the smallest positive $(1,1)$-entry of the matrices in $\widehat{\Gamma}_N M \widehat{\Gamma}_N$. By Lemma 2 including its notations we may assume that $M$ is upper triangular. Using Lemma 3 and the embedding \eqref{gl_5}, we may assume
\[
 L = \diag(\alpha^{\ast},m,\delta^{\ast}), \;\;\alpha^{\ast},\delta^{\ast} \in\N,\;\; \alpha^{\ast}\delta^{\ast} = m^2, \;\; \alpha^{\ast}\mid m\mid \delta^{\ast}.
\]
If we apply the Smith normal form to the upper left $2\times 2$ block and multiply by suitable matrices in \eqref{gl_4}, we obtain a matrix with $gcd(\alpha,\alpha^{\ast},a_1)$ as $(1,1)$-entry, if $a=(a_1,a_2,a_3)'$.
The choice of $\alpha$ and Lemma 2 lead to 
\[
 \alpha\mid \alpha^{\ast} \;\;\text{and}\;\; \alpha\mid a_1.
\]
If we multiply by a matrix $\widetilde{M}_G$, $G\in SL_2(\Z)$, in \eqref{gl_4'} from the right, we may replace $\alpha$ by $gcd(\alpha,a_3)$, thus $\alpha\mid a_3$. Multiplication by $M_{\lambda}$, $\lambda = \frac{1}{\alpha}(a_3,0,a_1)' \in\Z^3$, in \eqref{gl_2} from the right leads to 
\[
 a_1 = a_3 = c_1 = c_3 = 0.
\]
If we cancel the second and fourth row and column in $\frac{1}{m}M$, we obtain a matrix in $\Gcal_N$. Using this embedding of $\Gamma_N$ into $\widehat{\Gamma}_N$ instead of \eqref{gl_5}, Lemma 3 gives us the desired form. As the Smith invariants are unique even for all double cosets with respect to $GL_5(\Z)$, the result follows.
\end{proof}

\section{The Hecke algebra for the orthogonal groups $\Gamma_N$ and $\widehat{\Gamma}_N$}

First, we obtain a kind of multiplicativity in the Hecke algebras.

\begin{lemma}\label{Lemma_4} 
Let $N\in\N$ be squarefree and let $l,m\in\N$ be coprime.
\begin{enumerate}[a)]
\item Given $\frac{1}{l}L,\frac{1}{m}M\in\mathcal{G}_N$ with integral $L,M$, one has 
\[
 \Gamma_N\left(\frac{1}{l}L\right) \Gamma_N \cdot \Gamma_N\left(\frac{1}{m}M\right) \Gamma_N = \Gamma_N\left(\frac{1}{lm}LM\right) \Gamma_N.
\]
\item Given $\frac{1}{l}L,\frac{1}{m}M\in\widehat{\mathcal{G}}_N$ with integral $L,M$, one has 
\[
 \widehat{\Gamma}_N\left(\frac{1}{l}L\right) \widehat{\Gamma}_N \cdot \widehat{\Gamma}_N\left(\frac{1}{m}M\right) \widehat{\Gamma}_N = \widehat{\Gamma}_N\left(\frac{1}{lm}LM\right) \widehat{\Gamma}_N.
\]
\end{enumerate}
\end{lemma}

\begin{proof}
The double cosets are uniquely determined by the Smith invariants, which are multiplicative due to \cite{N}, Theorem II.5. Hence, only the double coset of $\frac{1}{lm}LM$ can occur in the product. The multiplicity is $1$ due to a standard argument (cf. \cite{K1}, Lemma V(6.1)).
\end{proof}

Given a prime $p$, i.e. $p\in\Pb$, we consider the subgroups
\begin{align*}
 \mathcal{G}_{N,p}: & = \left\{\frac{1}{p^r} K\in \mathcal{G}_N;\;K\;\text{integral}, \;r\in \N_0\right\},   \\
 \widehat{\mathcal{G}}_{N,p}: & = \left\{\frac{1}{p^r} M\in \widehat{\mathcal{G}}_N;\;M\;\text{integral}, \;r\in \N_0\right\}.
\end{align*}

\begin{theorem}\label{Lemma_2} 
If $N\in\N$ is squarefree, the Hecke algebras $\Hcal(\Gamma_N,\Gcal_N)$ and $\Hcal(\widehat{\Gamma}_N,\widehat{\Gcal}_N)$ are commutative and coincide with the tensor product of their primary components
\[
 \bigotimes_p\Hcal(\Gamma_N,\Gcal_{N,p}) \quad \text{resp.} \quad \bigotimes_p\Hcal(\widehat{\Gamma}_N,\widehat{\Gcal}_{N,p}).
\]
\end{theorem}

\begin{proof}
For $\Gamma_N$ the result is due to \cite{K2} in combination with \cite{K3}. 
According to Theorem 1, the mapping $M\mapsto M^{-1}$ is an involution which induces the identity on each double coset. Hence, the Hecke algebra is commutative by virtue of the standard argument (cf. \cite{K1}, Corollary I(7.2)). The tensor product decomposition is a consequence of Theorem 1 and Lemma 4.
\end{proof}

It follows from \cite{K2} and \cite{K3} that
\[
 \Hcal(\Gamma_N,\Gcal_{N,p}) = \Z\left[\Gamma_N,\Gamma_N \frac{1}{p}\diag\left(1,p,p^2\right) \Gamma_N \right].
\]
Now we define 
\begin{align*}
 T_{N,1}(p): & = \widehat{\Gamma}_N \frac{1}{p}\diag\left(1,p,p,p,p^2\right)\widehat{\Gamma}_N ,  \\
 T_{N,2}(p): & = \widehat{\Gamma}_N \frac{1}{p}\diag\left(1,1,p,p^2,p^2\right)\widehat{\Gamma}_N .
\end{align*}

Representatives of the right cosets in these double cosets can be computed from Lemma \ref{Lemma_2} with $m=p$.  
One just has to check the integrality of $M$ and $p^2M^{-1}= \widehat{S}^{-1}_N M'\widehat{S}_N$ in \eqref{gl_8} and then calculates the rank of these matrices over $\Z/p\Z$ according to Theorem \ref{Theorem_1}. 
For instance, the number of right cosets in $T_{N,1}(p)$ is
\begin{alignat*}{2}
& 1+p+p^2+p^3, &\quad & \text{if}\;\; p\nmid N,  \\
& p+2p^2+p^3,  &\quad & \text{if}\;\; p\mid N,
\end{alignat*}
and the number of right cosets in $T_{N,2}(p)$ is 
\begin{alignat*}{2}
& p+p^2+p^3+p^4, &\quad & \text{if}\;\; p\nmid N,  \\
& 2p^3+2p^4,  &\quad & \text{if}\;\; p\mid N,
\end{alignat*}
in accordance with \cite{RS2}, p. 192-193. Precise representatives in the symplectic setting are given in \cite{HK}.

\begin{theorem}\label{Theorem_3} 
Let $N\in\N$ be squarefree and $p$ a prime. The primary component $\Hcal(\widehat{\Gamma}_N,\widehat{\Gcal}_{N,p})$ consists of all polynomials in $T_{N,1}(p)$, $T_{N,2}(p)$, which are algebraically independent. $\Hcal(\widehat{\Gamma}_N,\widehat{\Gcal}_N)$ does not contain any zero-divisors.
\end{theorem}

\begin{proof}
Let $\Hcal^{(r)}_{N,p}$, $r\in \N_0$, denote the $\Z$-module spanned by all the double cosets
\[
 \widehat{\Gamma}_N\left(\frac{1}{p^r}M\right)\widehat{\Gamma}_N,\;\;\frac{1}{p^r}M\in\widehat{\Gcal}_{N,p},\;\; M\;\text{integral}.
\]
Hence, we have
\[
 \Hcal^{(0)}_{N,p} = \Z \widehat{\Gamma}_N, \;\; \Hcal^{(1)}_{N,p} = \Z\widehat{\Gamma}_N + \Z T_{N,1}(p) + \Z T_{N,2}(p)
\]
due to Theorem \ref{Theorem_1}. Now let
\[
 \widehat{\Gamma}_N\left(\frac{1}{p^r}M\right)\widehat{\Gamma}_N \in \Hcal^{(r)}_{N,p}\backslash \Hcal^{(r-1)}_{N,p}, \;\; r\geq 2.
\]
By Theorem \ref{Theorem_1} we may assume that
\[
 M = \diag (1,p^s,p^r, p^{2r-s},p^{2r}), \;\; 0\leq s\leq r.
\]
If $s\geq 1$, the $p$-rank of this matrix is $1$,
\begin{align*}
 & L: = \diag (1,p^{s-1},p^{r-1},p^{2r-2-s},p^{2r-2}),   \\
 & \widehat{\Gamma}_N \left(\frac{1}{p^{r-1}}L\right) \widehat{\Gamma}_N \in \Hcal^{(r-1)}_{N,p}.
\end{align*}
The same arguments as in \cite{K1}, Proposition V(8.1), show that
\[
 T_{N,1}(p) \cdot \widehat{\Gamma}_N\left(\frac{1}{p^{r-1}}L\right) \widehat{\Gamma}_N = \widehat{\Gamma}_N\left(\frac{1}{p^{r}}M\right) \widehat{\Gamma}_N + R,
\]
where $R\in \Hcal^{(r-1)}_{N,p}$. If $s=0$, let
\begin{align*}
  & L: = \diag (1,1,p^{r-1},p^{2r-2},p^{2r-2}),   \\
 & \widehat{\Gamma}_N \left(\frac{1}{p^{r-1}}L\right) \widehat{\Gamma}_N \in \Hcal^{(r-1)}_{N,p}.
\end{align*}
Just as above we conclude that
\[
 T_{N,2}(p)\cdot \widehat{\Gamma}_N\left(\frac{1}{p^{r-1}}L\right) \widehat{\Gamma}_N = \widehat{\Gamma}_N\left(\frac{1}{p^{r}}M\right) \widehat{\Gamma}_N + R_1 + R_2,
\]
where $R_1\in \Hcal^{(r-1)}_{N,p}$ and 
\[
 R_2 \in \sum^{r}_{s=1} \Z\widehat{\Gamma}_N \left(\frac{1}{p^r}\diag(1,p^s,p^r,p^{2r-s},p^{2r}\right) \widehat{\Gamma}_N.
\]
Hence, an induction shows that $T_{N,1}(p)$, $T_{N,2}(p)$ generate $\Hcal(\widehat{\Gamma}_N,\widehat{\Gcal}_{N,p})$. The number of double cosets in $\Hcal^{(r)}_{N,p}$ is $\binom{r+2}{2}$ due to Theorem \ref{Theorem_1}. A polynomial $T_{N,1}(p)^u\cdot T_{N,2}(p)^v$, $u,v\in \N_0$, belongs to $\Hcal^{(r)}_{N,p}$ if and only if $u+v\leq r$. As the number of these polynomials also equals $\binom{r+2}{2}$, the double cosets $T_{N,1}(p)$ and $T_{N,2}(p)$ are algebraically independent.
\end{proof}

\section{The Hecke algebra for the extended paramodular group $\Sigma^{\ast}_N$}

The symplectic group of degree $2$
\[
 \Sp_2(\R): = \{M\in \R^{4\times 4};\;J[M] = J\}, \;\; J = \begin{pmatrix}
                                                            0 & -E \\ E & 0
                                                           \end{pmatrix},
\]
acts on the Siegel half-space of degree $2$
\[
 H_2(\R):= \{Z = X+iY\in\C^{2\times 2};\; Z=Z',\,Y>0\}
\]
via
\[
 Z\mapsto M\langle Z\rangle = (AZ+B)(CZ+D)^{-1},\;\; M=\begin{pmatrix}
                                                        A & B \\ C & D
                                                       \end{pmatrix}.
\]
The orthogonal group $SO_0(\widehat{S}_N;\R)$ acts on the orthogonal half-space
\[
 \Hcal_N:= \{z=x+iy = (\tau_1,w,\tau_2)'\in\C^3;\;\mathrm{Im}\,\tau_1>0,\,S_N[y]>0\}
\]
via
\[
 z\mapsto \widetilde{M}\langle z\rangle:=\frac{1}{\widetilde{M}\{z\}} (-\tfrac{1}{2}S_N[z]b + K z+ c),
\]
where 
\[
 \widetilde{M} = 
 \begin{pmatrix}
  \alpha & a'S_N & \beta \\ b & K & c \\ \gamma & d'S_N & \delta                
 \end{pmatrix},\;\;
 \widetilde{M}\{z\} = - \tfrac{\gamma}{2} S_N[z] + d'S_N z + \delta.
\]
We consider the bijection between the complex symmetric $2\times 2$ matrices and $\C^3$
\[
 \phi_N:\begin{pmatrix}
         \alpha & \beta \\ \beta & \gamma
        \end{pmatrix}
\mapsto(\alpha,\beta,N\gamma)',
\]
which satisfies
\[
 \phi_N(H_2(\R) ) = \Hcal_N.
\]
There is an isomorphism between $\Sp_2(\R)/\{\pm E\}$ and $SO_0(\widehat{S}_N;\R)$, where $\pm M\mapsto \widetilde{M}$, given by 
\begin{gather*}\tag{9}\label{gl_9}
 \phi_N(M\langle Z\rangle) = \widetilde{M}\langle \phi_N(Z)\rangle\;\;\text{for all}\;Z\in H_2(\R)
\end{gather*}
(cf. \cite{G}, \cite{K3}). Note that \eqref{gl_2} and \eqref{gl_3} lead to
\[
 \widetilde{\begin{pmatrix}
                       E & S \\ 0 & E
                      \end{pmatrix}}
= M_{\lambda},\;\, \lambda=\phi_N(S), \;\, \widetilde{J}_N = J^*,\;\, J_N =
\begin{pmatrix}
 0 & -I^{-1} \\ I & 0
\end{pmatrix}, \;\, I = \begin{pmatrix}
			1 & 0 \\ 0 & N
			\end{pmatrix}.
\]
We obtain an explicit form of this isomorphism if we use the abbreviation
\[
 \begin{pmatrix}
  \alpha & \beta \\ \gamma & \delta
 \end{pmatrix}^{\sharp} = \begin{pmatrix}
			   \delta & -\beta \\ -\gamma & \alpha
			  \end{pmatrix}
\]
for the adjoint matrix, via $\phi_N(Z) = z$,
\[
 \det (CZ+D) = \widetilde{M}\{z\},\;\; (AZ+B)(CZ+D)^{\sharp} = \tfrac{1}{2} S_N[z]b+Kz+c,
\]
where
\begin{equation*}\label{gl_10}\tag{10}
 \begin{lcase}
 & \gamma = -(\det C)/N,\;\; d=\phi_N(C^{\sharp}D)/N,\quad\delta = \det D,   \\
 & \alpha = \det A,\;\, a = -N\phi_N(A^{\sharp}B),\;\, \beta = -N\det B,  \\
 & b = \frac{-1}{N} \phi_N (AC^{\sharp}),\;\; c= \phi_N (BD^{\sharp}), \;\; K=(r,s,t),  \\
 & \phi_N(AZD^{\sharp} + BZ^{\sharp}C^{\sharp}) = \tau_1 r + ws + N\tau_2 t = Kz.
 \end{lcase}
\end{equation*}
Note that the formulas for the first row of $\widetilde{M}$ may be obtained from the last row of $(\widetilde{J_N M}) = J^{\ast} \widetilde{M}$.

Let $\Sigma_N$ denote the (rational) paramodular group given by all the matrices in $\Sp_2(\Q)$ of the form
\begin{gather*}\label{gl_11}\tag{11}
 \begin{pmatrix}
   a_1 & a_2 N & b_1 & b_2 \\
   a_3 & a_4 & b_3 &  b_4/N  \\
   c_1 & c_2 N & d_1 & d_2  \\
   c_3 N & c_4 N & d_3 N & d_4
 \end{pmatrix},
\quad \text{where} \quad a_i,b_i,c_i,d_i \in \Z.
\end{gather*}
Note that for $M\in\Sigma_N$ the matrices
\begin{gather*}\label{gl_12neu}\tag{12}
 \begin{pmatrix}
  a_1 & b_1 \\ c_1 & d_1
 \end{pmatrix} \bmod{N}, \quad 
 \begin{pmatrix}
  a_4 & b_4 \\ c_4 & d_4
 \end{pmatrix} \bmod{N} \;\;\text{belong to}\;\; SL_2(\Z/N\Z).
\end{gather*}
The extended paramodular group $\Sigma^{\ast}_N$ is generated by $\Sigma_N$ and $W_d$, $d\mid N$, where
\begin{gather*}\label{gl_13neu}\tag{13}
 W_d = \begin{pmatrix}
        V_d & 0 \\ 0 & V_d'^{-1}
       \end{pmatrix},\,
 V_d = \frac{1}{\sqrt{d}} \begin{pmatrix}
                           \alpha d & \beta N \\ \gamma & \delta d
                          \end{pmatrix}, \,
 \alpha,\beta,\gamma,\delta \in \Z, \, \det V_d = 1.
\end{gather*}
It is a maximal discrete normal extension of $\Sigma_N$ (cf. \cite{Koe}).
Thus, we obtain

\begin{lemma}\label{Lemma_5} 
Let $N\in \N$ be squarefree. Then an isomorphism
\[
  \Sp_2(\R)/\{\pm E\} \to SO_0(\widehat{S}_N;\R),\quad \pm M \mapsto \widetilde{M},
\]
satisfying \eqref{gl_9} is given by \eqref{gl_10}. This isomorphism maps 
\begin{align*}
& \mathbb{G}: = \left\{\tfrac{1}{\sqrt{m}} M\in\Sp_2(\R);\; M\;\text{integral}, m\in\N\right\} \;\;\text{onto}\;\; SO_0(\widehat{S}_N;\Q),   \\[1ex]
& \Sigma^{\ast}_N \;\;\text{onto}\; \widehat{\Gamma}_N,  \\[1ex]
& \Sigma_N \;\text{onto}\; \{M = (m_{ij})\in\widehat{\Gamma}_N;\;m_{33} \equiv 1 \bmod{2N}\},\\
& \text{which is the discriminant kernel in}\; \widehat{\Gamma}_N.
\end{align*}
\end{lemma}

\begin{proof}
Apply \cite{K3}, Corollary 6, and \cite{GH2} for the last two groups. Clearly, the image of $\mathbb{G}$ is contained in $SO_0(\widehat{S}_N;\Q)$. As the latter group is generated by 
\[
 M_{\lambda}, \; \lambda \in \Q^3, \;\, J^{\ast},\;\, \diag(\alpha,1,1,1,1/\alpha),\,\alpha\in\Q,\;\alpha>0,\;\,\widehat{K},\; K\in SO_0(S_N;\Q),
\]
the surjectivity follows from \cite{K3}, Theorem 4.
\end{proof}

As the groups are isomorphic, this is true for the attached Hecke algebras as well (cf. \cite{K1}). Note that for $u,v \in \N$ 
\begin{gather*}\label{gl_14neu}\tag{14}
 M = \pm \tfrac{1}{\sqrt{u^2 v}}\diag(1,u,u^2 v,uv) \Rightarrow \widetilde{M} = \tfrac{1}{uv} \diag(1,v,uv,u^2v,u^2v^2).
\end{gather*}
Thus, we can reformate Theorem \ref{Theorem_1} and Lemma \ref{Lemma_5} as 

\begin{lemma}\label{Lemma_6} 
Let $N\in \N$ be squarefree. \\[1ex]
a) Given $M\in\G$, then the double coset $\Sigma^*_N M\Sigma^*_N$ contains a unique representative of the form
\[
 \tfrac{1}{\sqrt{u^2v}} \diag(1,u,u^2 v,uv), \;\; u,v\in\N.
\]
b) Given $L,M\in\G$ such that the common denominators of $\widetilde{L}$ and $\widetilde{M}$ are coprime, then
\[
 \Sigma^*_N L\Sigma^*_N \cdot \Sigma^*_N M \Sigma^*_N = \Sigma^*_N(LM)\Sigma^*_N.
\]
\end{lemma}

If we define the cases $uv=p$, $p\in \Pb$, in \eqref{gl_14neu} as
\begin{gather*}\label{gl_15neu}\tag{15}
\begin{split}
 \Tcal^{\ast}_{N,1} (p) : & = \Sigma^{\ast}_N \tfrac{1}{\sqrt{p}}\diag\left(1,1,p,p\right)\Sigma^{\ast}_N,  \\
 \Tcal^{\ast}_{N,2} (p) : & = \Sigma^{\ast}_N \tfrac{1}{p}\diag\left(1,p,p^2,p\right)\Sigma^{\ast}_N,
 \end{split}
\end{gather*}
the result of Theorem \ref{Theorem_3} can be reformulated as

\begin{theorem}\label{Theorem_4} 
Let $N\in\N$ be squarefree. Then the Hecke algebra $\Hcal(\Sigma^{\ast}_N,\mathbb{G})$ is commutative, does not contain any zero-divisors, and coincides with the tensor product of its primary components, which are the polynomials over $\Z$ in the algebraically independent elements $\Tcal^{\ast}_{N,1}(p)$ and $\Tcal^{\ast}_{N,2} (p)$, $p\in\Pb$.
\end{theorem}

\section{The Hecke algebra for the paramodular group $\Sigma_N$}

As an application we consider the Hecke algebra with respect to $\Sigma_N$. Note that 
\[
 \Sigma_N W_d \Sigma_N = \Sigma_N W_d = W_d \Sigma_N \;\; \text{for all}\;\; d\mid N,
\]
\begin{gather*}\label{gl_16neu}\tag{16}
 \Sigma_N W_d \Sigma_N \cdot \Sigma_N W_e\Sigma_N = \Sigma_N W_f \Sigma_N, \quad 
 f=\frac{de}{gcd(d,e)^2},\;\; d\mid N,\;\; e\mid N.
\end{gather*}
Hence, we have for all $M\in\mathbb{G}$
\begin{gather*}\label{gl_17neu}\tag{17}
 \Sigma^{\ast}_N M \Sigma^{\ast}_N = \bigcup_{d\mid N,\,e\mid N}\Sigma_N W_d M W_e \Sigma_N,
\end{gather*}
\begin{gather*}\label{gl_18neu}\tag{18} 
 \Sigma_N W_d M W_e \Sigma_N  = \Sigma_N W_d \Sigma_N \cdot \Sigma_N M \Sigma_N \cdot \Sigma_N W_e \Sigma_N.
\end{gather*}
If we denote the double cosets with respect to $\Sigma_N$ instead of $\Sigma^{\ast}_N$ in \eqref{gl_15neu} by $\Tcal_{N,1}(p)$ resp. $\Tcal_{N,2}(p)$, the result is

\begin{corollary}\label{Corollary_2} 
Let $N\in \N$ be squarefree. Then the Hecke algebra $\Hcal(\Sigma_N,\mathbb{G})$ is generated by the double cosets
\[
 \Sigma_N,\;\; \Sigma_N W_p\Sigma_N\; (p\in\Pb,\; p\mid N), \;\; \Tcal_{N,1}(q),\;\; \Tcal_{N,2}(q)\; (q\in \Pb).
\]
If $N>1$, then $\Hcal(\Sigma_N,\mathbb{G})$ fails to be commutative and contains zero-divisors.
\end{corollary}

\begin{proof}
Assume that the Hecke algebra is commutative and that $N> 1$. We may choose
\[
 V_N = \frac{1}{\sqrt{N}}\begin{pmatrix}
                          0 & N \\ -1 & 0
                         \end{pmatrix}.
\]
In view of $(\Sigma_N W_N \Sigma_N)^ 2 = \Sigma_N$ in \eqref{gl_16neu} this yields
\[
 \Sigma_N W_N M W_N \Sigma_N = \Sigma_N M \Sigma_N,
\]
i.e.
\begin{gather*}\label{gl_ast}
W_N M W_N L M^ {-1} \in \Sigma_N   \tag{\text{$\ast$}}
\end{gather*}
for some $L\in \Sigma_N$. If we set
\[
 M = \frac{1}{N} \diag(1,N,N^2,N)
\]
then \eqref{gl_ast} yields $L\in\Z^{4\times 4}$ and the second column belongs to $N\Z^4$. This contradicts $\det L = 1$. 

Moreover, note that \eqref{gl_16neu} yields
\[
 (\Sigma_N W_N\Sigma_N-\Sigma_N)\cdot(\Sigma_N W_N \Sigma_N+\Sigma_N) = 0.
\]
\vspace{-8ex}\\
\end{proof}

The result on the non-commutativity agrees with the fact that $T^{\ast}_{1,0} \neq T_{1,0}$, so that $T_{1,0}$ is not self-dual in \cite{RS2}, p. 194.

\bigskip
It follows from Lemma \ref{Lemma_6} and \eqref{gl_17neu} that each double coset $\Sigma_N M \Sigma_N$, $M\in \G$, contains a representative in block diagonal form
\begin{gather*}\tag{19}\label{gl_19neu}
\begin{pmatrix}
 A & 0 \\ 0 & D
\end{pmatrix}, \quad AD' = E.
\end{gather*}

If we denote by $\nu(M)$ the minimal $m\in\N$ such that $\sqrt{m}M$ is of the form \eqref{gl_11}, then \eqref{gl_16neu}, \eqref{gl_18neu}, \eqref{gl_19neu} and Corollary 3 in \cite{K2} applied to the $D$-block yield

\begin{lemma}\label{Lemma_7} 
Let $N\in\N$ be squarefree and $d\in\N$, $d\mid N$. Then the double coset $\Sigma_N W_d \Sigma_N$ commutes with the double cosets
\[
 \Sigma_N W_e \Sigma_N,\; e\mid N, \; \Sigma_N \tfrac{1}{\sqrt{m}}\diag(1,1,m,m)\Sigma_N, \;m\in\N,\;\Sigma_N M \Sigma_N,
\]
where $M\in\G$ such that $d$ and $\nu(M)$ are coprime.
\end{lemma}

We write $u\mid N^{\infty}$ if the prime divisors of $u$ divide $N$. Given $d\mid N$ and $u,v \in\N$, let 
\[
 u=u_1 u_2, \;\; u_1\mid d^{\infty}, \;\; gcd(du_1,u_2) = 1.
\]

As we may choose $\alpha,\delta\in du_1\Z$ and $\beta,\gamma \in u_2 \Z$ in \eqref{gl_13neu}, we get 
\begin{gather*}\label{gl_20neu}\tag{20}
W_d \diag(1,u,u^2 v,uv) W_d \in \Sigma_N \diag(u_1,u_2,u_1u^2_2 v,u^2_1u_2 v)\Sigma_N.
\end{gather*}

\begin{lemma}\label{Lemma_8} 
Let $N\in \N$ be squarefree and $M\in \G$. Then the double coset $\Sigma_N M\Sigma_N$ possesses a unique representative of the form
\[
 W_d \tfrac{1}{\sqrt{u^2_1 u^2_2 v}}\diag(u_1,u_2,u_1 u^2_2 v,u^2_1 u_2 v),
\]
\[
 d\mid N,\; u_1,u_2,v\in \N, \;\; u_1\mid N^{\infty},\;\; gcd(u_1,u_2) = 1,\;\; \nu(M) = du^2_1 u^2_2 v.
\]
\end{lemma}

\begin{proof}
The existence is a consequence of Lemma \ref{Lemma_6}, \eqref{gl_17neu} and \eqref{gl_20neu}. For the uniqueness we observe that $u=u_1u_2,v$ are unique due to Lemma \ref{Lemma_6} and $d$ in view of $\nu(M) = du^2v$. By virtue of \eqref{gl_12neu}, the rank of $\diag(u_1,u_1 u^2_2 v)$ over $\Z/p\Z$ is an invariant of the $\Sigma_N$-double coset for all $p\in \Pb$, $p\mid N$. This yields the uniqueness of $u_1$ and $u_2$. 
\end{proof}

Next, we derive a general form of multiplicativity.

\begin{lemma}\label{Lemma_9} 
Let $N\in\N$ be squarefree and $L,M\in \G$ such that $\nu(L)$ and $\nu(M)$ are coprime. Then
\[
\Sigma_N L \Sigma_N \cdot \Sigma_N M \Sigma_N = \Sigma_N L M \Sigma_N = \Sigma_N M \Sigma_N \cdot \Sigma_N L \Sigma_N.
\]
\end{lemma}

\begin{proof}
We choose representatives
\begin{align*}
 L & = W_d \tfrac{1}{\sqrt{u^2_1 u^2_2 v}} \diag(u_1,u_2,u_1 u^2_2 v,u^2_1 u_2 v), \;\; u= u_1 u_2,  \\
 M & = W_e \tfrac{1}{\sqrt{r^2_1 r^2_2 s}} \diag(r_1,r_2,r_1 r^2_2 s,r^2_1 r_2 s), \;\; r= r_1 r_2
\end{align*}
in the form of Lemma \ref{Lemma_8}.

First, let $d= e=1$. It follows from Lemma \ref{Lemma_6} that the product consists of double cosets in $\Sigma^*_N L M \Sigma^*_N$. As $\sqrt{u^2 v r^2 s}\,K$ for all the matrices $K$ in $\Sigma_N L \Sigma_N M \Sigma_M$ are of the form \eqref{gl_11}, representatives of the double cosets, which might occur, are given by
\[
 \tfrac{1}{\sqrt{u^2 v r^2 s}} \diag(t_1,t_2,t_1t^2_2 v s,t^2_1 t_2 v s), \;\; t_1\mid N^{\infty},\; gcd(t_1,t_2) = 1, \; t_1t_2 = ur
\]
due to \eqref{gl_20neu}. If follows from \eqref{gl_12neu} that for $U\in \Z^{2\times 2}$ with $\det U \equiv 1\bmod{N}$ the ranks of
\[
 \begin{pmatrix}
  u_1 & 0 \\ 0 & u_1u^2_2 v
 \end{pmatrix} U
 \begin{pmatrix}
  r_1 & 0 \\ 0 & r_1r^2_2 s
 \end{pmatrix}\; \text{and}\; 
  \begin{pmatrix}
  t_1 & 0 \\ 0 & t_1t^2_2 vs
 \end{pmatrix}
\]
over $\Z/p \Z$ coincide for all $p\mid N$. Thus, we have $t_1=u_1 r_1$ and $t_2=u_2 r_2$. The coefficient is $1$ due to the standard argument (cf. \cite{K1}, Lemma V(6.1)).

The case of general $d$ and $e$ can then be obtained from \eqref{gl_18neu}. Lemma \ref{Lemma_7} and \eqref{gl_16neu} imply
\[
 \Sigma_N L \Sigma_N \cdot \Sigma_N M \Sigma_N = \Sigma_N M \Sigma_N \cdot \Sigma_N L \Sigma_N = \Sigma_N W_{de} K \Sigma_N
\]
for some diagonal matrix $K$, since $du^2v$ and $er^2 s$ are coprime. Thus, the product is a single double coset and we may choose $M L$ or $LM$ as a representative.
\end{proof}

For a prime $p$ we define the primary component 
\[
 \G_p: = \{M\in \G;\; \nu(M) = p^r,\,r\in \N_0\}.
\]

%
%
%

The analog of Theorem \ref{Theorem_4} is

\begin{theorem}\label{Theorem_5} 
Let $N\in \N$ be squarefree. Then the Hecke algebra $\Hcal(\Sigma_N,\G)$ is the tensor product of the primary components $\Hcal(\Sigma_N,\G_p)$, $p\in\Pb$. One has
\[
 \Hcal(\Sigma_N,\G_p) = 
 \begin{cases}
  \Z[\Tcal_{N,1}(p), \Tcal_{N,2}(p)], \quad & \text{if}\;\; p\nmid N, \\
  \Z[\Sigma_N W_p \Sigma_N,\Tcal_{N,1}(p),\Tcal_{N,2}(p)] \quad & \text{if}\;\; p\mid N.
 \end{cases}
\]
The primary component $\Hcal(\Sigma_N,\G_p)$ is commutative if and only if $p\nmid N$.
\end{theorem}

\begin{proof}
According to Lemma \ref{Lemma_9} and Corollary \ref{Corollary_2}, the result on the commutativity remains to be proved. If $p\nmid N$ we may choose the same representatives of the right cosets as in the case $N=1$, where the result follows from the standard case (cf. \cite{An}, \cite{Fr}, \cite{K1}).

The same calculations as above show that
\[
 \Sigma_N W_p \Sigma_N \cdot \Tcal_{N,2}(p) \cdot \Sigma_N W_p \Sigma_N = \Sigma_N\tfrac{1}{p} \diag(p,1,p,p^2)\Sigma_N,
\]
which is different from $\Tcal_{N,2}(p)$ for $p\mid N$ due to Lemma \ref{Lemma_8}. Hence we have
\[
 \Sigma_N W_p \Sigma_N \cdot T_{N,2} (p) \neq T_{N,2}(p) \cdot \Sigma_N W_p \Sigma_N.
\]
\vspace*{-8ex}\\
\end{proof}

We add

\begin{remarkson}
a) We have for all $p\mid N$
\[
  \Sigma_N W_p \Sigma_N \cdot T_{N,1} (p) = T_{N,1}(p) \cdot \Sigma_N W_p \Sigma_N
\]
according to \eqref{gl_18neu}. One can prove along the lines of Lemma \ref{Lemma_9} that 
\[
 \Tcal_{N,1}(p)\cdot \Tcal_{N,2}(p) = \alpha \Tcal_{N,1}(p) + \beta\Sigma_N \tfrac{1}{\sqrt{p^3}} \diag(1,p,p^3,p^2)\Sigma_N
\]
for some $\alpha,\beta \in\N$. All the double cosets are invariant under $M\mapsto M^{-1}$. Hence, the induced anti-homomorphism from \cite{K1}, Theorem I(7.1), yields
\[
 \Tcal_{N,1}(p) \cdot \Tcal_{N,2}(p) = \Tcal_{N,2}(p)\cdot \Tcal_{N,1}(p)\;\; \text{for all} \;\; p\in \Pb.
\]
b) Considering $Nr^2$, $r\in \N$, the paramodular group $\Sigma_{Nr^2}$ is conjugate to a subgroup of $\Sigma_N$ of finite index. The same is true for the orthogonal groups in section 2 an 3.
\end{remarkson}
\bigskip
\noindent
\small{The authors thank the anonymous referee for helpful comments.}

\vspace{8ex}

\normalsize

\renewcommand\refname{References}

\end{document}